\documentclass[12pt]{article}
\usepackage{wasysym}  
\usepackage[width=7in,height=9.5in,center]{crop}
\usepackage{mathtools,amssymb,mathrsfs,amsthm,mathrsfs}
\usepackage{contour}
\usepackage{accents}
\usepackage{calc}
\usepackage{lmodern}
\usepackage{tikz}
\usepackage{tkz-euclide}
\usetkzobj{all}
\usepackage{microtype}
\usepackage{colonequals}
\usepackage{color}
\usepackage[pdftitle={Pfaffian formulas for spanning tree probabilities},
            pdfauthor={Greta Panova and David B. Wilson},
            bookmarks=true,bookmarksopen=true,bookmarksopenlevel=2]{hyperref}

\makeatletter
\newif\if@borderstar%
\def\bordermatrix{\@ifnextchar*{%
\@borderstartrue\@bordermatrix@i}{\@borderstarfalse\@bordermatrix@i*}%
}
\def\@bordermatrix@i*{\@ifnextchar[{\@bordermatrix@ii}{\@bordermatrix@ii[()]}}
\def\@bordermatrix@ii[#1]#2{%
\begingroup
\m@th\@tempdima8.75\p@\setbox\z@\vbox{%
\def\cr{\crcr\noalign{\kern 2\p@\global\let\cr\endline }}%
\ialign {$##$\hfil\kern 2\p@\kern\@tempdima & \thinspace %
\hfil $##$\hfil && \quad\hfil $##$\hfil\crcr\omit\strut %
\hfil\crcr\noalign{\kern -\baselineskip}#2\crcr\omit %
\strut\cr}}%
\setbox\tw@\vbox{\unvcopy\z@\global\setbox\@ne\lastbox}%
\setbox\tw@\hbox{\unhbox\@ne\unskip\global\setbox\@ne\lastbox}%
\setbox\tw@\hbox{%
$\kern\wd\@ne\kern -\@tempdima\left\@firstoftwo#1%
\if@borderstar\kern2pt\else\kern -\wd\@ne\fi%
\global\setbox\@ne\vbox{\box\@ne\if@borderstar\else\kern 2\p@\fi}%
\vcenter{\if@borderstar\else\kern -\ht\@ne\fi%
\unvbox\z@\kern-\if@borderstar2\fi\baselineskip}%
\if@borderstar\kern-2\@tempdima\kern2\p@\else\,\fi\right\@secondoftwo #1 $%
}\null \;\vbox{\kern\ht\@ne\box\tw@}%
\endgroup
}
\makeatother

\newcommand{\old}[1]{}
\newcommand{\Pf}{\operatorname{Pf}}
\newcommand{\cross}{\operatorname{cr}}
\newcommand{\sign}{\operatorname{sign}}
\renewcommand{\L}{\mathscr{L}}
\newcommand{\A}{\mathscr{A}}
\newcommand{\M}{\mathscr{M}}
\newcommand{\Z}{\mathscr{Z}}
\newcommand{\Zs}{\accentset\star\Z}
\newcommand{\U}[1]{\overset{#1}{\textsf{U}}}
\newcommand{\D}[1]{\overset{#1}{\textsf{D}}}
\newcommand{\F}[1]{\overset{#1}{\textsf{F}}}
\newcommand{\I}[1]{\overset{#1}{\textsf{I}}}
\newcommand{\s}{\textsf{S}}
\renewcommand{\S}[1]{\overset{#1}{\textsf{S}}}
\newcommand{\Op}{\Osc{\oplus}}
\newcommand{\Om}{\Osc{\ominus}}
\newcommand{\Oc}{\Osc{\ocircle}}
\newcommand{\Od}{\Osc{\odot}}
\newcommand{\Osc}[1]{\mathchoice{\scalebox{1.2}{$#1$}}{\scalebox{1.2}{$#1$}}{\scalebox{0.83333}{$#1$}}{\scalebox{.6}{$#1$}}}
\newcommand{\OP}[1]{\overset{#1}{\Op}}
\newcommand{\OM}[1]{\overset{#1}{\Om}}
\newcommand{\OD}[1]{\overset{#1}{\Od}}
\newcommand{\OC}[1]{\overset{#1}{\Oc}}
\renewcommand{\SS}[1]{\OC{#1}\OD{#1}}
\newcommand{\II}[1]{\OD{#1}\OC{#1}}

\newcommand{\C}{\mathbb{C}}
\newcommand{\G}{\mathcal G}
\renewcommand{\L}{\mathscr{L}}
\newcommand{\Gv}{\mathscr{G}}

\newtheorem{theorem}{Theorem}
\newtheorem{lemma}[theorem]{Lemma}
\newtheorem{corollary}[theorem]{Corollary}

\numberwithin{theorem}{section}

\newcommand{\fm}{\phantom{-}}

\newcommand{\ijNN}[1]{{\mathrlap{\raisebox{12pt}{$\scriptstyle j=1,\ldots,#1$}}i=1,\dots,#1}}
\newcommand{\ijnn}{\ijNN{n}}

\begin{document}
\title{\mbox{Pfaffian formulas for spanning tree probabilities}}
\author{\begin{tabular}{c}\href{http://www.math.ucla.edu/~panova/}{Greta Panova}\\[-4pt]\small University of Pennsylvania\end{tabular} \and \begin{tabular}{c}\href{http://dbwilson.com}{David B. Wilson}\\[-4pt]\small Microsoft Research\end{tabular}}
\date{}
\renewcommand{\thefootnote}{\fnsymbol{footnote}}
\renewcommand{\thefootnote}{\arabic{footnote}}

\maketitle
\thispagestyle{empty}
\begin{abstract}
We show that certain topologically defined uniform spanning tree
probabilities for graphs embedded in an annulus can be computed as
linear combinations of Pfaffians of matrices involving the line-bundle
Green's function, where the coefficients count cover-inclusive Dyck
tilings of skew Young diagrams.
\end{abstract}

\section{Introduction}

In \cite{KW4} it was shown that the probabilities of topologically defined uniform spanning tree events can be computed as linear combinations of determinants of matrices whose entries involve the Green's function $G$ and the derivative $G'$ of the ``line-bundle Green's function''.  These probabilities were used to compute the intensity of loop-erased random walk \cite{KW4} and the probabilities of local events in the abelian sandpile model \cite{wilson:sand-local}.  We give another formula involving Pfaffians.  In addition to being (somewhat) computationally more efficient, the Pfaffian formula implies some structural properties of the polynomials in $G$ and $G'$.  First, it becomes apparent that the coefficients of the polynomials are integers --- previously the coefficients were only known to be half-integers.  Second, if each $G'_{u,v}$ variable is replaced with $G'_{u,v}+f(u)-f(v)$, then each polynomial is unchanged.  This invariance property was observed earlier for small sizes, but a general proof was missing until now, and it simplifies some of the sandpile calculations.

For background on the line-bundle Laplacian, response matrix, and Green's function for graphs embedded in surfaces, and their use in computing spanning tree probabilities, we refer the reader to \cite{KW4}.  Here we summarize the key facts that we use.

The line-bundle Green's function $\Gv_{u,v}(z)$ is a generalization of the usual Green's function $G_{u,v}$, where $u$ and $v$ are vertices of a graph $\G$, and $z\in\C$.  When $z=1$ it specializes to the usual Green's function:
\[
G_{u,v}=G_{v,u} =\Gv_{u,v}(1)\,.
\]
The line-bundle Green's function $\Gv$ has the symmetry $\Gv_{v,u}(z)=\Gv_{u,v}(1/z)$.  We define
\[
G'_{u,v}=\left[\frac{d}{dz} \Gv_{u,v}(z)\right]_{z=1}\,,
\]
which is antisymmetric, and is what we referred to as the derivative of the line-bundle Green's function.  From the symmetry of $\Gv$ it follows that $G'$ is antisymmetric: $G'_{v,u}=-G'_{u,v}$.
For the Green's function there is a designated sink vertex $s$ (which has been suppressed from the notation) for which
\[ G_{u,s} = 0 \]
for each vertex $u$.

There is another set of electrical variables that are useful to work with, the response matrix, or the Dirichlet-to-Neumann matrix $L_{u,v}$.  The response matrix is defined with respect to a designated set of vertices which we call \textbf{nodes}.  We think of the nodes as being ``boundary vertices'', and the other vertices as being internal, and the response matrix gives the linear map from voltages to current flows.  In the line bundle setting we denote the response matrix by $\L_{u,v}(z)$.  Here too $\L_{v,u}(z)=\L_{u,v}(1/z)$, and the line-bundle response matrix specializes to the usual response matrix when $z=1$,
\[
L_{u,v}=L_{v,u}=\left[\L_{u,v}(z)\right]_{z=1}\,,
\]
which is symmetric, and we define
\[
L'_{u,v}=\left[\frac{d}{dz} \L_{u,v}(z)\right]_{z=1}\,
\]
which is antisymmetric.  The response matrix variables satisfy the additional relation
\[ \sum_v L_{u,v} = 0
\]
for each vertex $u$.

Spanning tree probabilities for a graph $\G$ embedded on an annulus can be computed in terms of either set of variables, $\{G_{u,v},G'_{u,v}\}$ or $\{L_{u,v},L'_{u,v}\}$.

Suppose that graph $\G$ has $n$ nodes, which we label $\{1,\dots,n\}$.  A \textbf{grove} is a forest such that each tree contains at least one node.  Groves were first studied by Carroll and Speyer \cite{CS}, and then more systematically by Kenyon and Wilson \cite{KW1}, who gave this current definition.  Any grove induces a set partition $\sigma$ on the nodes where each set consists of the nodes from the same tree.  We let $Z[\sigma]$ denote the weighted sum of groves whose induced partition is $\sigma$.  The weighted sum of spanning trees $Z[\text{tree}]=Z[1,\dots,n]$ can be computed via the matrix-tree theorem, so we are interested in computing the ratios
\begin{equation} \label{eq:Zdot}
\dddot Z[\sigma] \colonequals \frac{Z[\sigma]}{Z[1|2|\cdots|n]}
\end{equation}
or
\begin{equation} \label{eq:Zbar}
\overline{Z}[\sigma] \colonequals \frac{Z[\sigma]}{Z[\text{tree}]}.
\end{equation}

Suppose that the graph $\G$ is embedded in an annulus so that the nodes $1,\dots,n-1$ are arranged in cyclic order on one boundary of the annulus, while node $n$ is on the other boundary of the annulus.  The grove partition function ratios \eqref{eq:Zdot} and \eqref{eq:Zbar} for these ``\textbf{annular-one graphs}'' $\G$ were used to compute probabilities for loop-erased random walk \cite{KW4} and for recurrent sandpile configurations \cite{wilson:sand-local}.  For annular-one graphs, $\dddot Z[\sigma]$ can be expressed in terms of a linear combination of determinants involving $L$ and $L'$, while $\overline{Z}_\sigma$ can be expressed in terms of a linear combination of determinants involving $G$ and $G'$ \cite{KW4}.  We shall re-express them as linear combinations of Pfaffians.

It turns out that $Z[\sigma]$ is itself a linear combination of $Z[\tau]$'s, where each $\tau$ is a ``partial pairing'' of the nodes $1,\dots,n$ \cite{KW4}.  A \textbf{partial pairing} is a set of pairs of nodes, singletons, and ``internalized'' nodes, which are not listed in the partition, but which may appear in any of the parts (like the other non-node vertices in a grove).  For example,
\begin{multline*}
   Z[2,6,9|3,4,5|7|1,8] = Z[2,9|3,5|7|1,8] - Z[2,9|3,6|7|1,8] \\- Z[2,9|3,5|6,7|1,8] - Z[2,9|3,5|1,6|7] + Z[2,9|3,5|1,6|7,8]\,.
\end{multline*}
It turns out that for the LERW and sandpile applications it suffices to assume that node $n$ is in a doubleton part.

Kenyon and Wilson \cite{KW4} showed that $Z[\tau]/Z[1|2|\cdots|n]$ can
be expressed as a linear combination of determinants involving the
$L_{i,j}$'s and $L'_{i,j}$'s, and that $Z[\tau]/Z[1,2,\cdots,n]$ can
be expressed as a linear combination of determinants involving the
$G_{i,j}$'s and $G'_{i,j}$'s.  We will give these determinant formulas
in the next section, since they are the starting point of the present work.

\subsection{Partial pairings in terms of Pfaffians}

For a partial pairing $\tau$ in which node $n$ is in a doubleton part,
we can encode $\tau$
by a string $\lambda$ of $n$ symbols, where the symbol at position $i$ encodes the role of node~$i$ in the partial pairing.
For bookkeeping purposes that will soon become apparent,
we label each symbol with the label of the node that it represents; when the labels
are $1,\dots,n$ we sometimes omit the labels.
For example, for the annular partial pairing
\[\tau=1,5|2|3,4|7,10|8,14|11|12,13 \quad=\quad
\begin{tikzpicture}[baseline=0.0cm-2.5pt]
\draw[fill=gray!40!white,draw=none] (0,0) circle(1.5);
\draw[fill=white,draw=none] (0,0) circle(.4);
    \foreach \x in {1,...,13} { \coordinate (\x) at ({1.5*cos((\x-0.5)*360/13)},{1.5*sin((\x-0.5)*360/13)});}
    \coordinate (14) at ({360*15/25}:.4);
    \draw[very thick] (14)--(8);
    \tkzDefPoint({-360/13}:1.5){1213};
    \tkzDefPoint({360*3/13}:1.5){34};
    \tkzDrawArc[very thick,color=black](1213,13)(12);
    \tkzDrawArc[very thick,color=black](34,4)(3);
    \draw[very thick] (1) .. controls +({360*13.5/26}:.5) and +({360*22.5/26}:.5) .. (5);
    \draw[very thick] (7) .. controls (65:1.2) and (15:1.2) .. (10);
    \foreach \x in {1,2,3,4,5,7,8,10,11,12,13,14} {\node [circle,fill=white,draw=black,inner sep=0.5pt] at (\x) {$\phantom{3}\mathclap{\x}\phantom{3}$};}
\end{tikzpicture}
\]
the associated (labeled) encoding string is
\begin{equation} \label{lambda}
\lambda =\lambda(\tau) = \U{1}\S{2}\U{3}\D{4}\D{5}\I{6}\D{7}\F{8}\I{9}\U{10}\S{1\!1}\U{12}\D{13}\OD{14}\,.
\end{equation}

Here node $n$, which is on the other boundary, is given the special symbol $\Od$.  The node paired with $n$ is also given a special symbol, \textsf{F}.  (So each $\lambda(\tau)$ has exactly one $\F{}$ and one $\Od$ symbol.)  \textsf{I} indicates that the node has been internalized, and \textsf{S} indicates a node in a singleton part.  The remaining nodes are assigned the symbols \textsf{U} and \textsf{D} so that when the \textsf{F} is cyclically rotated to the end, the substring $\lambda^\circ$ formed by the \textsf{U}'s and \textsf{D}'s defines a (labeled) Dyck path whose associated noncrossing matching is the pairing of the nodes in $\tau$.
In the above example,
\[ \lambda^\circ = \U{10}\U{12}\D{13}\U{1}\U{3}\D{4}\D{5}\D{7}
\quad=\quad
\begin{tikzpicture}[scale=0.4,baseline=0.0cm-1pt]
\draw[thick] (0,0) -- (1,1) -- (2,2) -- (3,1) -- (4,2) -- (5,3) -- (6,2) -- (7,1) -- (8,0);
\draw[fill,draw=none] (0,0) circle(.1);
\draw[fill,draw=none] (1,1) circle(.1);
\draw[fill,draw=none] (2,2) circle(.1);
\draw[fill,draw=none] (3,1) circle(.1);
\draw[fill,draw=none] (4,2) circle(.1);
\draw[fill,draw=none] (5,3) circle(.1);
\draw[fill,draw=none] (6,2) circle(.1);
\draw[fill,draw=none] (7,1) circle(.1);
\draw[fill,draw=none] (8,0) circle(.1);
\node at (0.5,0.5) {\contour{white}{$\scriptstyle 10$}};
\node at (1.5,1.5) {\contour{white}{$\scriptstyle 12$}};
\node at (2.5,1.5) {\contour{white}{$\scriptstyle 13$}};
\node at (3.5,1.5) {\contour{white}{$\scriptstyle 1$}};
\node at (4.5,2.5) {\contour{white}{$\scriptstyle 3$}};
\node at (5.5,2.5) {\contour{white}{$\scriptstyle 4$}};
\node at (6.5,1.5) {\contour{white}{$\scriptstyle 5$}};
\node at (7.5,0.5) {\contour{white}{$\scriptstyle 7$}};
\end{tikzpicture}
\quad=\quad
\begin{tikzpicture}[scale=0.4,baseline=0.0cm-1pt]
\draw[thick] (0,0) -- (1,1) -- (2,2) -- (3,1) -- (4,2) -- (5,3) -- (6,2) -- (7,1) -- (8,0);
\draw[fill,draw=none] (0,0) circle(.1);
\draw[fill,draw=none] (1,1) circle(.1);
\draw[fill,draw=none] (2,2) circle(.1);
\draw[fill,draw=none] (3,1) circle(.1);
\draw[fill,draw=none] (4,2) circle(.1);
\draw[fill,draw=none] (5,3) circle(.1);
\draw[fill,draw=none] (6,2) circle(.1);
\draw[fill,draw=none] (7,1) circle(.1);
\draw[fill,draw=none] (8,0) circle(.1);
\draw[thick,color=gray] (0.5,0.5)--(7.5,0.5);
\draw[thick,color=gray] (1.5,1.5)--(2.5,1.5);
\draw[thick,color=gray] (3.5,1.5)--(6.5,1.5);
\draw[thick,color=gray] (4.5,2.5)--(5.5,2.5);
\node at (0.5,0.5) {\contour{white}{$\scriptstyle 10$}};
\node at (1.5,1.5) {\contour{white}{$\scriptstyle 12$}};
\node at (2.5,1.5) {\contour{white}{$\scriptstyle 13$}};
\node at (3.5,1.5) {\contour{white}{$\scriptstyle 1$}};
\node at (4.5,2.5) {\contour{white}{$\scriptstyle 3$}};
\node at (5.5,2.5) {\contour{white}{$\scriptstyle 4$}};
\node at (6.5,1.5) {\contour{white}{$\scriptstyle 5$}};
\node at (7.5,0.5) {\contour{white}{$\scriptstyle 7$}};
\end{tikzpicture}
\,.
\]
We call the string $\lambda$ the \textbf{augmented cyclic Dyck path} associated with the partial pairing~$\tau$ --- ``augmented'' because it contains symbols not in the Dyck path $\lambda^\circ$, and ``cyclic'' because its start is determined by the location of the $\F{}$ symbol.

Given two labeled augmented cyclic Dyck paths $\lambda$ and $\mu$, we say that $\lambda\preceq\mu$ if they are the same length, have the same labels,
all the letters other than $\U{}$ and $\D{}$ are the same in both $\lambda$ and $\mu$, and as Dyck paths, $\lambda^\circ$ lies below $\mu^\circ$.

If $\lambda$ is a \textbf{labeled string}, we let $\mathbf\lambda_i$ denote its $i$th labeled symbol, and we let $\lambda(i)$ denote the label of $\lambda_i$.

For a labeled augmented cyclic Dyck path~$\mu$, we define $\mu^{\I{}}$ to be the labeled string obtained from~$\mu$ by deleting all the $\s$ letters, and replacing each $\I{i}$ with the two letters $\II{i}$.  We also define $\mu^\s$ to be the labeled string obtained from $\mu$ by deleting all the $\I{}$ letters, and replacing each $\S{i}$ with the two letters $\SS{i}$.  For example, if $\mu$ is the labeled augmented cyclic Dyck path in \eqref{lambda}, then
\[
\mu^{\I{}} = \U{1}\U{3}\D{4}\D{5}\II{6}\D{7}\F{8}\II{9}\U{10}\U{12}\D{13}\OD{14}
\]
and
\[
\mu^\s = \U{1}\SS{2}\U{3}\D{4}\D{5}\D{7}\F{8}\U{10}\SS{1\!1}\U{12}\D{13}\OD{14}\,.
\]
Next we define $\dddot\mu$ and $\overline\mu$.  Recall that there is only one letter $\F{}$ in $\mu$; let $f$ be its label, so that $\F{f}\in\mu$.
For each letter $\U{i}$, $\D{i}$, $\F{f}$ in $\mu^\s$, we make the substitutions
\begin{align*}
 \U{i} \mapsto \begin{cases} \OP{i} & i<f \\ \OC{i} & i>f \end{cases}
\quad\quad\quad\quad
 \D{i} \mapsto \begin{cases} \OC{i} & i<f \\ \OM{i} & i>f \end{cases}
\quad\quad\quad\quad
 \F{f} \mapsto \OC{f}
\end{align*}
to obtain $\overline\mu$.
We let $\dddot\mu$ be the result of these same substitutions applied to $\mu^{\I{}}$.
For our example,
\[
\dddot\mu = \OP{1}\OP{3}\OC{4}\OC{5}\II{6}\OC{7}\OC{8}\II{9}\OC{10}\OC{12}\OM{13}\OD{14}
\]
and
\[
\overline\mu = \OP{1}\SS{2}\OP{3}\OC{4}\OC{5}\OC{7}\OC{8}\OC{10}\SS{1\!1}\OC{12}\OM{13}\OD{14}\,.
\]
The original string $\mu$ can be recovered from either $\overline\mu$ or $\dddot\mu$.

Given a string $\sigma$ of $m$ labeled symbols $\Op$, $\Om$, $\Oc$, $\Od$, such as as the ones above, we define an $m\times m$ matrix $M_\sigma(A,A')$ by
\[
M_\sigma(A,A') \colonequals \bordermatrix[{[]}]{
&\scriptstyle\sigma_j\neq\Od & \scriptstyle\sigma_j=\Od \cr
\scriptstyle\sigma_i\neq\Od &
-A'_{\sigma(i),\sigma(j)} + A_{\sigma(i),\sigma(j)}
\left(\!\begin{aligned}+1_{\sigma_i=\Op}-1_{\sigma_j=\Op}\\[-3pt]
-1_{\sigma_i=\Om}+1_{\sigma_j=\Om}\end{aligned}\right)
 &
A_{\sigma(i),\sigma(j)}
\cr
\scriptstyle\sigma_i=\Od &
-A_{\sigma(i),\sigma(j)}
& 0
}_{\ijNN{m}\,.}
\]
The symbols $\Op$, $\Om$, and $\Oc$ are mnemonic for $+1$, $-1$, and $0$, which
go into the coefficient of $A_{\sigma(i),\sigma(j)}$ when $\sigma_i,\sigma_j\neq\Od$.
For example, when $\sigma$ is the above value for~$\overline\mu$, this matrix is
\[
\scalebox{0.7}{
\bordermatrix[{[]}]{
&\overset{-}{1}&2&\overset{\Od}{2}&\overset{-}{3}&4&5\ \ 7\ \ 8\ \ 10\ \ 11&\overset{\Od\,}{11}&12&\overset{+}{13}&\overset{\Od\,}{14}\cr
\mathllap{{\scriptstyle{}^+}\,1}&
  0 & A_{1,2}{-}A'_{1,2} & A_{1,2} & -A'_{1,3} & A_{1,4}{-}A'_{1,4}&\cdots\cdots\cdots&A_{1,11}&A_{1,12}{-}A'_{1,12}&2A_{1,13}{-}A'_{1,13}&A_{1.14}\cr
\mathllap{2}&
A'_{1,2}{-}A_{1,2}&0&A_{2,2}&-A_{2,3}{-}A'_{2,3}&-A'_{2,4}&\cdots\cdots\cdots&A_{2,11}&-A'_{2,12}&A_{2,13}{-}A'_{2,13}&A_{2,14}\cr
\mathllap{\raisebox{1pt}{$\scriptstyle\Od$}\,2}&
-A_{1,2}&-A_{2,2}&0&-A_{2,3}&-A_{2,4}&\cdots\cdots\cdots&0&-A_{2,12}&-A_{2,13}&0 \cr
\mathllap{{\scriptstyle{}^+}\,3}&
\vdots&\vdots&\ddots&0&A_{3,4}{-}A'_{3,4}&\cdots\cdots\cdots&A_{3,11}&A_{3,12}{-}A'_{3,12}&2A_{3,13}{-}A'_{3,13}&A_{3,14}\cr
\mathllap{4}&
&&&&0&\cdots\cdots\cdots&A_{4,11}&-A'_{4,12}&A_{4,13}{-}A'_{4,13}&A_{4,14}\cr
\mathllap{5}&
&&&&&\ddots&\vdots&\vdots&\vdots&\vdots\cr
\mathllap{7}\cr
\mathllap{\vdots}
}}
\]

We define $M_\mu^{(L)}=M_{\dddot\mu}(L,L')$, and $M_\mu^{(G)}=M_{\overline\mu}(G,G')$,
where $G_{i,n}$ is replaced with~$1$.
Both $M_\mu^{(L)}$ and $M_\mu^{(G)}$ are antisymmetric.  The new formulas involve Pfaffians of these matrices $M^{(L)}_\mu$ and $M^{(G)}_\mu$.

\begin{figure}[b!]
\centerline{\includegraphics[width=\textwidth]{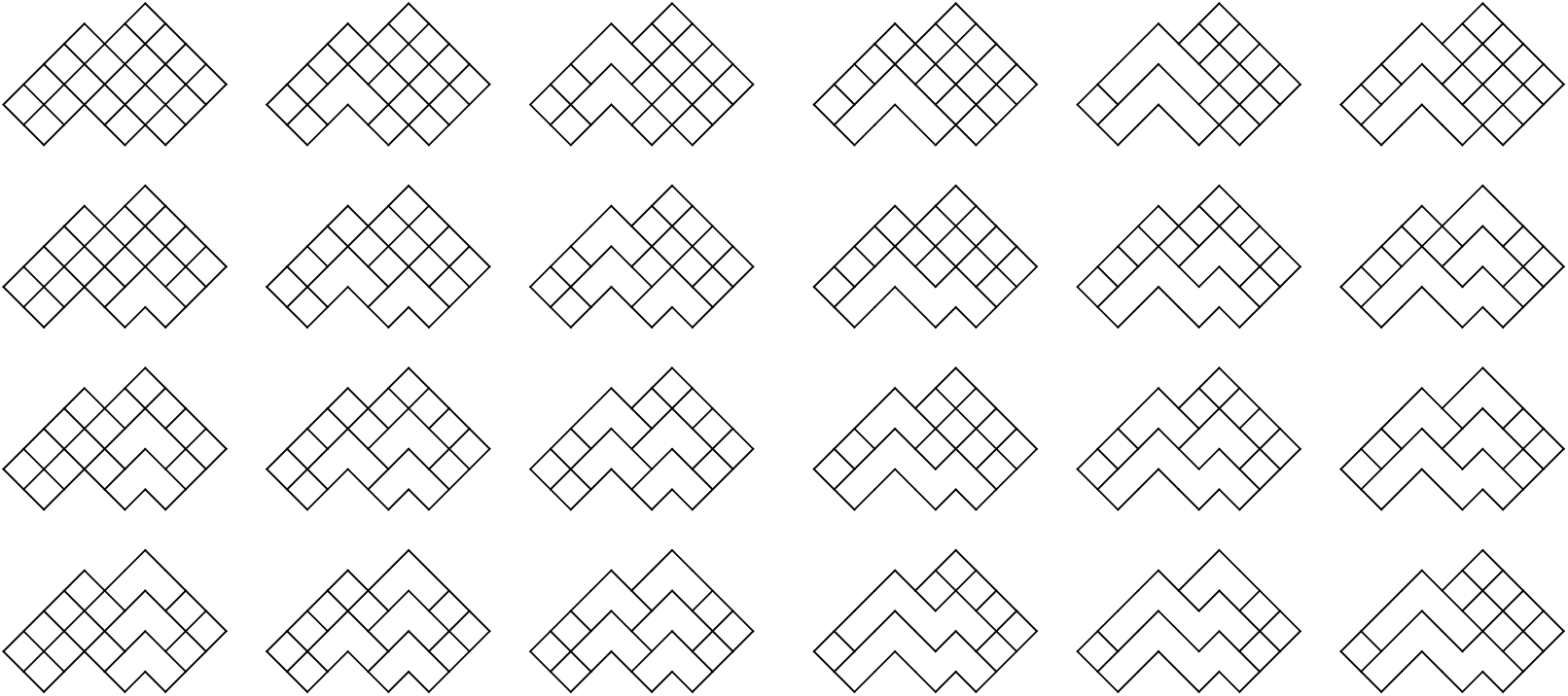}}
\caption{Cover-inclusive Dyck tilings of a skew shape.  (This figure first appeared in \cite{KW3}.)}
\label{fig:dyck-tilings}
\end{figure}

The new formulas have coefficients that are defined in terms of ``\textbf{cover-inclusive Dyck tilings}'', which were first defined in \cite{KW3} and independently in \cite{MR2927185}, and were studied further in \cite{kim,KMPW,KW4,KW5,JVK},
and whose definition we now recall.
If $\lambda$ and $\mu$ are Dyck paths such that $\lambda$ is below $\mu$, then the region $\lambda/\mu$ is a skew Young diagram (rotated $45^\circ$).  A Dyck tile is a ribbon tile which is shaped like a Dyck path, i.e., a collection of $\sqrt{2}\times\sqrt{2}$ boxes rotated $45^\circ$ centered at the points of a Dyck path.  A Dyck tiling of $\lambda/\mu$ is a tiling of it by Dyck tiles.  We say that one Dyck tile covers another Dyck tile if it contains a box which is directly (not diagonally) above a box of the other tile.  A cover-inclusive Dyck tiling is one for which, whenever a Dyck tile $T_1$ covers another Dyck tile $T_2$, the range of $x$-coordinates of $T_1$ is a subset of the range of $x$-coordinates of $T_2$.  See Figure~\ref{fig:dyck-tilings} for a list of the Dyck tilings of a particular skew shape $\lambda/\mu$.

\begin{theorem} \label{thm:Pf-udfe}
Suppose $\tau$ is a partial pairing of the nodes of an annular-one graph with $n$ nodes, where node $n$ is paired in $\tau$.  Let $\lambda$ be the labeled augmented cyclic Dyck path which encodes $\tau$.  Then
\begin{equation} \label{dot-sum-PfL}
\frac{Z[\tau]}{Z[1|2|\cdots|n]} = \sum_{\mu\succeq\lambda} [\text{\# of c.i.\ Dyck tilings of $\lambda^\circ/\mu^\circ$}] \times \Pf M^{(L)}_\mu
\end{equation}
and
\begin{equation} \label{bar-sum-PfG}
\frac{Z[\tau]}{Z[1,2,\dots,n]} = \sum_{\mu\succeq\lambda} [\text{\# of c.i.\ Dyck tilings of $\lambda^\circ/\mu^\circ$}] \times \Pf M^{(G)}_\mu\,.
\end{equation}
\end{theorem}

\subsection{Examples}
\enlargethispage{12pt}
We give a couple of examples:

For the partial pairing $1,3|2,4$,
the encoding string $\lambda$ is $\D{1}\F{2}\U{3}\OD{4}$,
the only $\mu$ in the sum is $\mu=\D{1}\F{2}\U{3}\OD{4}$,
for which the skew Young diagram $\lambda^\circ/\mu^\circ$ has only the empty Dyck tiling, so the coefficient is $1$.  For this $\mu=\D{1}\F{2}\U{3}\OD{4}$, $\dddot\mu=\OC{1}\OC{2}\OC{3}\OD{4}$, so
\[
  \frac{Z[1,3|2,4]}{Z[1|2|3|4]} =
\Pf\underbrace{\begin{bmatrix}
0 & -L_ {1,2}' & -L_ {1,3}' & L_{1,4} \\
 \fm L_ {1,2}' & 0 & -L_ {2,3}' & L_{2,4} \\
 \fm L_ {1,3}' & \fm L_ {2,3}' & 0 & L_{3,4} \\
 -L_{1,4} & -L_{2,4} & -L_{3,4} & 0
\end{bmatrix}}_{M^{(L)}_{\textsf{DFU}\Od}}
= -L'_{1,2} L_{3,4} - L'_{2,3} L_{1,4} - L'_{3,1} L_{2,4}\,,
\]
which matches \cite[eqn.~5.5b]{KW4}, and $\overline\mu=\OC{1}\OC{2}\OC{3}\OD{4}$, so
\[
  \frac{Z[1,3|2,4]}{Z[1,2,3,4]} =
\Pf\underbrace{\begin{bmatrix}
 0 & -G_ {1,2}' & -G_ {1,3}' & 1 \\
 \fm G_ {1,2}' & 0 & -G_ {2,3}' & 1 \\
 \fm G_ {1,3}' & \fm G_ {2,3}' & 0 & 1 \\
 -1 & -1 & -1 & 0
\end{bmatrix}}_{M^{(G)}_{\textsf{DFU}\Od}}
= -G'_{1,2} - G'_{2,3} - G'_{3,1}\,,
\]
which matches \cite[eqn.~5.6b]{KW4}.

For the partial pairing $1,4|2|6,7$ the encoding string is $\lambda=
\U{1}\S{2}\I{3}\D{4}\I{5}\F{6}\OD{7}$,
the only $\mu\succeq\lambda$ is
$\mu=\U{1}\S{2}\I{3}\D{4}\I{5}\F{6}\OD{7}$,
and $\dddot\mu=\OP{1}\II{3}\OC{4}\II{5}\OC{6}\OD{7}$, so
\[
\frac{Z[1,4|2|6,7]}{Z[1|2|3|4|5|6|7]}
= \Pf\underbrace{\scalebox{.79}{$\begin{bmatrix}
 0 & L_{1,3} & L_{1,3}-L_ {1,3}' & L_{1,4}-L_ {1,4}' & L_{1,5} &L_{1,5}-L_ {1,5}' & L_{1,6}-L_ {1,6}' & L_{1,7} \\
 -L_{1,3} & 0 & -L_{3,3} & -L_{3,4} & 0 & -L_{3,5} & -L_{3,6} & 0 \\
 L_ {1,3}'-L_{1,3} & L_{3,3} & 0 & -L_ {3,4}' & L_{3,5} & -L_ {3,5}'& -L_ {3,6}' & L_{3,7} \\
 L_ {1,4}'-L_{1,4} & L_{3,4} & L_ {3,4}' & 0 & L_{4,5} & -L_ {4,5}' &-L_ {4,6}' & L_{4,7} \\
 -L_{1,5} & 0 & -L_{3,5} & -L_{4,5} & 0 & -L_{5,5} & -L_{5,6} & 0 \\
 L_ {1,5}'-L_{1,5} & L_{3,5} & L_ {3,5}' & L_ {4,5}' & L_{5,5} & 0 &-L_ {5,6}' & L_{5,7} \\
 L_ {1,6}'-L_{1,6} & L_{3,6} & L_ {3,6}' & L_ {4,6}' & L_{5,6} & L_{5,6}' & 0 & L_{6,7} \\
 -L_{1,7} & 0 & -L_{3,7} & -L_{4,7} & 0 & -L_{5,7} & -L_{6,7} & 0 \\
\end{bmatrix}$}}_{M^{(L)}_{\textsf{USIDIF}\Od}}
\]
while $\overline\mu=\OP{1}\SS{2}\OC{4}\OC{6}\OD{7}$, so
\[
\frac{Z[1,4|2|6,7]}{Z[1,2,3,4,5,6,7]}
= \Pf\underbrace{\scalebox{1}{$\displaystyle\begin{bmatrix}
 0 & G_{1,2}-G_ {1,2}' & G_{1,2} & G_{1,4}-G_ {1,4}' & G_{1,6}-G_{1,6}' & 1 \\
 G_ {1,2}'-G_{1,2} & 0 & G_{2,2} & -G_ {2,4}' & -G_ {2,6}' & 1 \\
 -G_{1,2} & -G_{2,2} & 0 & -G_{2,4} & -G_{2,6} & 0 \\
 G_ {1,4}'-G_{1,4} & G_ {2,4}' & G_{2,4} & 0 & -G_ {4,6}' & 1 \\
 G_ {1,6}'-G_{1,6} & G_ {2,6}' & G_{2,6} & G_ {4,6}' & 0 & 1 \\
 -1 & -1 & 0 & -1 & -1 & 0 \\
\end{bmatrix}$}}_{M^{(G)}_{\textsf{USIDIF}\Od}}
\]

For $1,2|3,7|4,6$ we have $\lambda=\textsf{UDFUID}\Od$, there are two $\mu$'s such that $\mu\succeq\lambda$:
\[
 \frac{Z[1,2|3,7|4,6]}{Z[1,2,3,4,5,6,7]}
= \Pf M^{(G)}_{\textsf{UDFUID}\Od} + \Pf M^{(G)}_{\textsf{DDFUIU}\Od}
\]

For the partial pairing $1,3|2|4,10|5,6|7,9$ we have $\lambda=\textsf{USDFUDUID}\Od$, there are five $\mu$'s such that $\mu\succeq\lambda$, and for one of these $\mu$'s the skew Young diagram $\lambda^\circ/\mu^\circ$ has two Dyck tilings:
\begin{multline*}
 \frac{Z[1,3|2|4,10|5,6|7,9]}{Z[1,2,3,4,5,6,7,8,9,10]}
= \Pf M^{(G)}_{\textsf{USDFUDUID}\Od} +\Pf M^{(G)}_{\textsf{USDFUUDID}\Od}\\ +\Pf M^{(G)}_{\textsf{DSDFUDUIU}\Od} +\Pf M^{(G)}_{\textsf{DSDFUUDIU}\Od} +2\times\Pf M^{(G)}_{\textsf{DSDFUUUID}\Od}
\end{multline*}

\subsection{Corollaries}

The formulas in Theorem~\ref{thm:Pf-udfe} immediately imply the following statement.
\begin{corollary}
For a partition $\tau$ on $\{1,\ldots,n\}$ in which $n$ is not in a singleton part, on an annular-one graph with $n$ nodes,
the ratio $\frac{Z[\tau]}{Z[1|2|\cdots|n]}$ is a polynomial in the variables $L$ and $L'$ with integer coefficients.  Similarly, $\frac{Z[\tau]}{Z[1,2,\cdots,n]}$ is a polynomial in $G$ and $G'$ with integer coefficients.
\end{corollary}

It was known that these ratios are polynomials in the $L$ and $L'$ variables, or the $G$ and $G'$ variables \cite{KW4}, but the integrality of the coefficients was previously a mystery.

Recalling
\[
  \frac{Z[1,3|2,4]}{Z[1,2,3,4]} = -G'_{1,2} - G'_{2,3} - G'_{3,1}\,,
\]
observe that this polynomial is invariant under the substitution $G'_{i,j}\to G'_{i,j}+f(i)-f(j)$.  The next corollary states that this is a general phenomenon for the $G$-$G'$-polynomials of any partition:
\begin{corollary} \label{coboundary}
For a partition $\tau$ on $\{1,\ldots,n\}$ in which $n$ is not in a singleton part, on an annular-one graph with $n$ nodes,
the $G$-$G'$-polynomial for $\frac{Z[\tau]}{Z[1,2,\cdots,n]}$ is invariant under replacing each $G'_{i,j}$ with $G'_{i,j}+f(i)-f(j)$.
\end{corollary}

\begin{proof}[Proof of Corollary~\ref{coboundary}]
Consider each Pfaffian in the formula from Theorem~\ref{thm:Pf-udfe}. Since the last column (row) is all $1$'s ($-1$'s), we can add an all-$f(i)$'s row to row $i$ and subtracting an all-$f(i)$'s-column from column $i$, without changing the value of the Pfaffian.  Since $G'_{i,j}$ occurs only in row $i$ and column $j$ (and row $j$ and column $i$), with coefficient $1$ (and $-1$), these operations replace each $G'_{i,j}$ with $G'_{i,j}+f(i)-f(j)$ and keep the Pfaffian invariant.
\end{proof}

We remark that it was known \cite{KW4} that substituting $G'_{i,j}\to G'_{i,j}+f(i)-f(j)$ and then evaluating the polynomial at the values of $G$ and $G'$ that arise from an annular-one graphs will give a result independent of $f$.  Corollary~\ref{coboundary} is a stronger statement, since it was  not known whether the values of $G$ and $G'$ that arise from annular-one graphs are full-dimensional or whether
they satisfy algebraic relations which cause the substituted $G$-$G'$-polynomials, when evaluated at these values, to be independent of $f$.

\section{Determinant formulas}
\newcommand{\B}{\mathbf{B}}

For an annular partial pairing $\tau$ on $n$ nodes, let $\lambda$ be its encoding string, let $T$ denote the set of internalized nodes, and $Q$ denote the set of singleton nodes.
Let $k$ denote the order of the Dyck path $\lambda^\circ$, i.e., half its length, so that $n=2k+2+|Q|+|T|$.

Let $S\subset\{1,\dots,n\}\setminus(Q\cup T)$ be a subset of the paired nodes which has size $k+1$ and includes $n$, and let $R = \{1,\dots,n\}\setminus(S\cup Q\cup T)$
be the complementary set of paired nodes.
Given $\lambda$ and $S$, Kenyon and Wilson \cite{KW4} defined
\begin{multline} \label{eq:A-inverse}
  \B_{\lambda,S}(\zeta) = \sum_{\mu\succeq\lambda} [\text{\# of c.i.\ Dyck tilings of $\lambda/\mu$}] \times \zeta^{\text{\# indices in $S$ at which $\mu$ has an up-step}
} \times \\
  \zeta^{-\text{\# indices in $S\setminus\{n\}$ after $\lambda$'s flat step}
   \,+\, \text{\# down steps of $\lambda$ after $\lambda$'s flat step}}\,,
\end{multline}
and showed how to use these polynomials $\B_{\lambda,S}$ to compute the ratios of grove partition functions.  Specifically
\begin{equation} \label{dot-sum-detL(z)}
  \frac{Z[\tau]}{Z[1|2|\cdots|n]} = (-1)^{|T|}\times\lim_{z\to 1} \sum_{R,S} \frac{\B_{\lambda,S}(z^2)}{(1-z^2)^k} \det \L_{R,T}^{S,T}\,,
\end{equation}
where $\L_{R,T}^{S,T}$ denotes the submatrix of $\L$ whose rows are indexed by $R$ and $T$ and whose columns are indexed by $S$ and $T$,
and we need to specify a pairing between the indices of $R$ and $S$ to determine the signs of the determinants.  We use the Dvoretzky-Motzkin cycle lemma bijection to make this pairing, as indicated below (figure taken from \cite{KW4}).  Essentially we make a path with period $2k+1$ which has an up step at each index in $R$ and a down step at each index in $S\setminus\{n\}$.  The up and down steps are the endpoints of chords underneath the path, and these chords define the pairing, where the extra up step is paired with $n$.
\[
 {}_{R=\{3,4,6,7,8,11\}}^{S=\{1,2,5,9,10,12\}}
 \!\Rightarrow\! \raisebox{-12pt}{\includegraphics[scale=0.43]{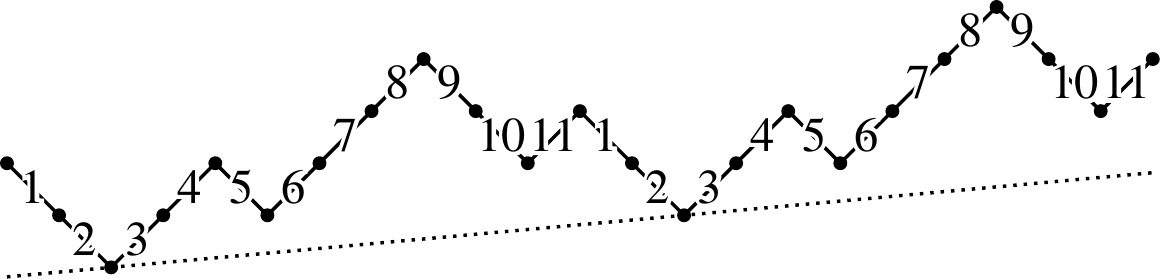}}
 \!\Rightarrow\! \raisebox{-12pt}{\includegraphics[scale=0.43]{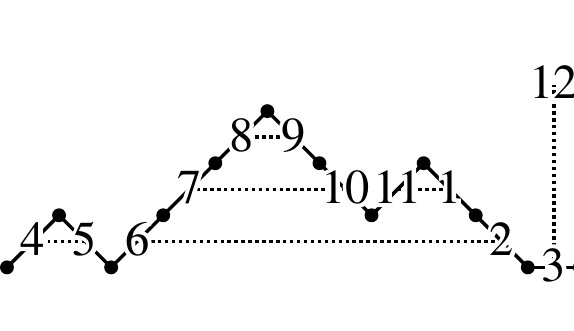}}
 \!\Rightarrow\!
 {}_4^5|{}_6^2|{}_{\;7}^{10}|{}_8^9|{}_{11}^{\,1}|{}_{\;3}^{12}
\]

Recall that $\L_{i,j}=\L_{i,j}(z)$ is a function of $z$.
We change variables to $z=e^t$ (here we differ slightly from the notation
in \cite{KW4}, which used $\zeta=z^2=e^t$).
We expand $\L_{i,j}(e^t) = L_{i,j}+ L'_{i,j} t + \cdots$, and
let $\hat\L_{i,j}$ denote its linearized approximation $\hat\L_{i,j} = L_{i,j}+ L'_{i,j} t$.  In general the series expansion for $\L_{i,j}(e^t)$ will have more terms, but while
it is not \textit{a priori\/} obvious, the limit \eqref{dot-sum-detL(z)} can be evaluated using $\hat\L_{i,j}$ in place of $\L_{i,j}(e^t)$:
\begin{equation} \label{dot-sum-detL}
  \frac{Z[\tau]}{Z[1|2|\cdots|n]} = (-1)^{|T|}\times\lim_{t\to0} \frac{1}{(-2t)^k} \sum_{R,S} \B_{\lambda,S}(e^{2t}) \det \hat\L_{R,T}^{S,T}\,,
\end{equation}
and a similar formula
\begin{equation} \label{bar-sum-detG}
  \frac{Z[\tau]}{Z[1,2,\ldots,n]} = \lim_{t\to0} \frac{1}{(-2t)^k} \sum_{R,S} \B_{\lambda,S}(e^{2t}) \det \hat\Gv_{R,Q}^{S,Q}\,,
\end{equation}
holds, where $\hat\Gv_{i,j}=G_{i,j}+ G'_{i,j} t$ and each $\hat\Gv_{i,n}$ is replaced with $1$ \cite{KW4}.  From these formulas we derive the Pfaffian formulas.

\section{Pfaffian formulas}

We start in section~\ref{sec:pf-det} by showing that a Pfaffian can be expressed as a sum of determinants.  In section~\ref{sec:tripartite} we give an application of this identity to tripartite pairings.  Then we use the Pfaffian identity and equations \eqref{dot-sum-detL} and \eqref{bar-sum-detG} to prove Theorem~\ref{thm:Pf-udfe} in section~\ref{sec:main}.

\subsection{The Pfaffian as a sum of determinants} \label{sec:pf-det}

For any matching $M=(i_1,j_1),\ldots,(i_k,j_k)$, we define $\sign(M)=(-1)^{\cross(M)}$, where $\cross(M)$ is the number of crossings of arcs from $M$ when $M$ is drawn as $k$ arcs between the points $\{1,\ldots,2k\}$ on a line.  For the left endpoint of each arc we can associate an up-step, and for each right endpoint we can associate a down-step, which results in a Dyck path.  The down steps of the matching $M$ are the down steps of its Dyck path, i.e., $\{\max(i_1,j_1),\dots,\max(i_k,j_k)\}$.

Given a set of positive integers $R$ for which $R\subset\{1,\dots,2|R|\}$, we define $d_R$ as follows.  We let $n=2|R|$ and $S=\{1,\dots,n\}\setminus R$.  For an arbitrary matrix $A$ we define
\begin{equation} \label{dR(M)}
d_R(A) \colonequals \begin{cases}
 \det[A_{i,j}]_{i\in R}^{j\in S} & n\in S \\
 \det[-A_{i,j}]_{i\in R}^{j\in S} & n\in R\,,\end{cases}
\end{equation}
 where $R$ and $S$ are ordered according to the Dvoretzky-Motzkin bijection as described above.  For example,
\[
d_{\{3,4,6,7,8,11\}}(A) = \det[A_{i,j}]_{i=4,6,7,8,11,3}^{j=5,2,10,9,1,12}\,.
\]

\begin{lemma}\label{lemma:det_expansion_S}
Suppose $n\geq 0$ is even, $R\subset\{1,\ldots,n\}$, $|R|=n/2$, and $S=\{1,\ldots,n\}\setminus R$.
Let $A$ be an arbitrary $n\times n$ matrix.
Then
\[ d_R(A) = \sum_{\substack{\text{directed matchings $M$ s.t.}\\\text{$M$ matches $R$ to $S$}}} (-1)^{\cross(M)} \prod_{(r,s)\in M} (-1)^{1_{r>s}}A_{r,s} \,.\]
\end{lemma}

\begin{proof}
Let $k=n/2$.
The arrangement of elements of $R$ and $S$ in $d_R$
is given by the Dvoretzky-Motzkin cycle lemma bijection, and in particular corresponds to a matching
$M_0=\{(r_1,s_1),\dots,(r_k,s_k)\}$ (each $r_\ell\in R$ and $s_\ell\in S$) which has no crossings when drawn in the annulus.
By the determinant expansion,
\begin{equation}\label{det_expansion}
\det A_R^S=\sum_{\pi \in \mathfrak{S}_k} \sign(\pi)\prod_{\ell=1}^k A_{r_\ell,s_{\pi(\ell)}}\,.
\end{equation}
Suppose $n\notin R$.  When we draw matching $M_0$ on a line, there may be crossings of the arc $(j,n)$ from arcs $(a,b)$, such that $a>j>b$;
these are precisely the arcs whose starting point is larger than its endpoint.
When drawn on the line, the number of crossings is $\cross(M_0)=\sum_{(i,j)\in M_0} 1_{i>j}$.
If instead $n\in R$, then $\cross(M_0)=\sum_{(i,j)\in M_0} 1_{i<j}$.

For a permutation $\pi$ let the matching $M(\pi)$ be $M(\pi)=\{(r_1,s_{\pi(1)}),\ldots,r_k,s_{\pi(k)}\}$.
The matching $M_0$ corresponds to the identity permutation,  so at least when the permutation $\pi$ is the identity, we have
\[\sign(\pi) = (-1)^{\cross(M(\pi))} (-1)^{ \sum_{(r,s)\in M(\pi)} 1_{r>s}} (-1)^{(n/2)1_{n\in R}}\,,\]
a formula which we now verify for the other permutations.
Any permutation $\pi$ can be expressed as a sequence of transpositions, and it is a straightforward case analysis to verify
that any transposition changes the parity of the number of crossings in the matching plus the number of arcs directed backwards.
\end{proof}

\begin{theorem} \label{Pf}
Suppose $n\geq 0$ is even.
If $A$ is an arbitrary $n\times n$ matrix, and $d_R(A)$ is as defined in \eqref{dR(M)}, then
\begin{equation} \label{sum-dR=Pf(A)}
\sum_{\substack{R\subset\{1,\dots,n\}\\|R|=n/2}} d_R(A) =
\Pf \big[A-A^T\big]\,,
\end{equation}
where $A^T$ is the transpose of $A$.
\end{theorem}
\begin{proof}
From Lemma~\ref{lemma:det_expansion_S}, we see that the left-hand side of~\eqref{sum-dR=Pf(A)} equals
\[\sum_{\text{directed matchings $M$}} (-1)^{\cross(M)} \prod_{(r,s)\in M} (-1)^{1_{r>s}}A_{r,s}\,.\]

Let $n=2k$, and let $W_{i,j}=A_{i,j}-A_{j,i}$.
We can expand the Pfaffian as
\begin{equation*}\label{pfaffian_expansion}
\Pf[W] =  \sum_{\substack{\text{undirected matchings $M$}\\M=\{(i_1,j_1),\ldots,(i_k,j_k)\}\\i_1<j_1,\dots,i_k<j_k\\j_1<\cdots<j_k}} (-1)^{\cross(M)} \prod_{\ell=1}^k W_{i_\ell,j_\ell}.
\end{equation*}
When we make the substitution $W_{i,j}=A_{i,j}-A_{j,i}$, this has the effect of choosing directions for each pairing, converting the sum over undirected matchings into a sum over directed matchings:
\[ \Pf\big[A-A^T\big] = \sum_{\text{directed matchings $M$}} (-1)^{\cross(M)} \prod_{(r,s)\in M} (-1)^{1_{r>s}} A_{r,s}\,.\qedhere\]
\end{proof}

\subsection{Applications of the Pfaffian identity} \label{sec:tripartite}

Before continuing with our main result, we mention an interesting consequence of Theorem~\ref{Pf}.  Curtis, Ingerman, and Morrow \cite{CIM} gave an interpretation of the determinant $\det L_R^S$ when $R=\{r_1,\ldots,r_k\}$ and $S=\{s_1,\ldots,s_k\}$ are disjoint subsets of $\{1,\ldots,n\}$, which, when translated into the language of groves, asserts that
\begin{equation}\label{CIM}
  \det L_{r_1,\ldots,r_k}^{s_1,\ldots,s_k} = \sum_{\pi\in\mathfrak{S}_k} \sign(\pi)\frac{Z[r_1,s_{\pi(1)}|\cdots|r_k,s_{\pi(k)}|\text{(other nodes singletons)}]}{Z[1|\cdots|n]}\,.
\end{equation}
This formula holds for any graph.

If $B$ and $C$ are two disjoint sets of nodes, and we set
$A_{i,j}=0$ when $i\in C$ or $j\in B$ and otherwise set $A_{i,j}=L_{i,j}$,
then Theorem~\ref{Pf} with the above interpretation of the minors
implies $\Pf[A-A^T]$ is a sum over directed matchings for which the nodes in $B$ are sources and the nodes in $C$ are destinations, of the sign of the directed matching times the grove ratio associated with that matching.  In particular, nodes of $B$ are only paired with nodes not in $B$, and nodes in $C$ are only paired with nodes not in $C$.  If a matching $M$ contains a pair $(i,j)$ where $i,j\in B$ or $i,j\in C$, then $M$ is not included in the sum.  Notice that if $i,j\notin B\cup C$, then the matching $(M\setminus\{(i,j)\})\cup\{(j,i)\}$, in which the pair $(i,j)$ has been reversed, has the same weight as $M$ but opposite sign.
Thus
\begin{multline}\label{L[RGB]}
\Pf\bordermatrix[{[]}]{& \scriptstyle j\in B & \scriptstyle j\notin B\cup C & \scriptstyle j\in C\cr
\scriptstyle i\in B & 0 & \fm L_{i,j} & \fm L_{i,j} \cr
\scriptstyle i\notin B\cup C & -L_{i,j} & 0 & \fm L_{i,j} \cr
\scriptstyle i\in C & -L_{i,j} & -L_{i,j} & 0 \cr
} =\\[8pt]=
 \sum_{{\substack{\text{directed matchings $M$}\\\text{if $(i,j)\in M$ then}\\\text{$i\in B$ or $j\in C$ or both}}}}
\left( (-1)^{\cross(M)} \prod_{(i,j)\in M} (-1)^{1_{j<i}}\right)
 \dddot{Z}[i_1,j_1|\cdots|i_{n/2},j_{n/2}]\,.
\end{multline}
When the graph is circular planar (i.e., the nodes lie on the outer face of a planar graph), and $B=\{1,\ldots,|B|\}$ and $C=\{n+1-|C|,\ldots,n\}$, there is only one matching~$M$ for which $\dddot{Z}[M]\neq 0$, and the sign is positive, so $\dddot{Z}[M]$ is the Pfaffian.
For example,
\[
\dddot{Z}\left[
\begin{tikzpicture}[baseline=0.0cm-2.5pt]
\draw[fill=gray!40!white,draw=none] (0,0) circle(0.8);
    \foreach \x in {1,...,6} { \coordinate (\x) at ({0.8*cos((\x-0.5)*360/6)},{0.8*sin((\x-0.5)*360/6)});}
    \draw[very thick] (1) arc(120:240:{0.8/sqrt(3.0)});
    \draw[very thick] (3) arc(240:360:{0.8/sqrt(3.0)});
    \draw[very thick] (5) arc(0:120:{0.8/sqrt(3.0)});
    \node [circle,fill=white!80!red,draw=red,inner sep=0.5pt] at (1) {1};
    \node [circle,fill=white!80!red,draw=red,inner sep=0.5pt] at (2) {2};
    \node [circle,fill=white!80!green,draw=green,inner sep=0.5pt] at (3) {3};
    \node [circle,fill=white!80!green,draw=green,inner sep=0.5pt] at (4) {4};
    \node [circle,fill=white!80!blue,draw=blue,inner sep=0.5pt] at (5) {5};
    \node [circle,fill=white!80!blue,draw=blue,inner sep=0.5pt] at (6) {6};
\end{tikzpicture}
\right]=
\frac{Z[1,6|2,3|4,5]}{Z[1|2|3|4|5|6]} =
\Pf\begin{bmatrix}
    0     &    0     &\fm L_{1,3}&\fm L_{1,4}&\fm L_{1,5}&\fm L_{1,6}\\
    0     &    0     &\fm L_{2,3}&\fm L_{2,4}&\fm L_{2,5}&\fm L_{2,6}\\
   -L_{1,3}&   -L_{2,3}&    0     &    0     &\fm L_{3,5}&\fm L_{3,6}\\
   -L_{1,4}&   -L_{2,4}&    0     &    0     &\fm L_{4,5}&\fm L_{4,6}\\
   -L_{1,5}&   -L_{2,5}&   -L_{3,5}&   -L_{4,5}&    0     &    0     \\
   -L_{1,6}&   -L_{2,6}&   -L_{3,6}&   -L_{4,6}&    0     &    0
\end{bmatrix}\,.
\]
This is one of several tripartite matching formulas that
were derived earlier by Kenyon and Wilson \cite{KW2} using a different method \cite{KW1}.

The determinant formula \eqref{CIM} has been extended in several directions.  Kenyon and Wilson \cite{KW4} showed that if $Q=\{q_1,\ldots,q_\ell\}$ and $T=\{t_1,\ldots,t_m\}$,
and $Q,R,S,T$ partition $\{1,\ldots,n\}$, then
\begin{equation} \label{directed-grove}
  \det \L_{r_1,\ldots,r_k,t_1,\ldots,t_m}^{s_1,\ldots,s_k,t_1,\ldots,t_m}
   = (-1)^m \sum_{\pi\in\mathfrak{S}_k} \sign(\pi)\frac{\Z\big[{}_{r_1}^{s_{\pi(1)}}|\cdots|{}_{r_k}^{s_{\pi(k)}}|q_1|\cdots|q_\ell\big]}{\Z[1|\cdots|n]}\,,
\end{equation}
where the $\Z$'s give the weighted sum of ``cycle-rooted groves''.  (The cycle weights go to zero and $\Z$ converges to $Z$ when $z\to 1$, see \cite{KW4} for further explanation.)
When we combine Theorem~\ref{Pf} with the above formula, we obtain
the following:
\begin{theorem}
Suppose there are $n$ nodes, $P,Q,T$ partition $\{1,\ldots,n\}$, and $|P|=2k$ is even.  For each $i\in P$ let $\alpha_i$ and $\beta_i$ be parameters, and for $i\in T$ let $\alpha_i=\beta_i=1$.  List the nodes $p_1,\ldots,p_{2k},t'_1,t_1,\ldots,t'_m,t_m$, where $t'_i$ is a second copy of $t_i$, and let $T'=\{t'_1,\ldots,t'_m\}$.  Then
\newpage
\begin{multline}\label{Pf(LL)}
\Pf
\bordermatrix[{[]}]{& \scriptstyle j\in P\cup T & \scriptstyle j\in T'\cr
\scriptstyle i\in P\cup T &
\alpha_i\beta_j\L_{i,j}-\alpha_j\beta_i \L_{j,i} &
\alpha_i\L_{i,j} \cr
\scriptstyle i\in T' &
-\alpha_j \L_{j,i} &
0 \cr
}_{{\mathrlap{\raisebox{12pt}{$\scriptstyle j=p_1,\ldots,p_{2k},t'_1,t_1,\ldots,t'_m,t_m$}}i=p_1,\ldots,p_{2k},t'_1,t_1,\ldots,t'_m,t_m}}
=\\=
\sum_{\substack{\text{\rm directed matchings $M$ of $P$}\\M=\{(r_1,s_1),\ldots,(r_k,s_k)\}}} \left((-1)^{\cross(M)} \prod_{(r,s)\in M} (-1)^{1_{s<r}} \alpha_r\beta_s\right) \times\frac{\Z[{}_{r_1}^{s_{1}}|\cdots|{}_{r_k}^{s_{k}}|q_1|\cdots|q_\ell]}{\Z[1|\cdots|n]}\,.
\end{multline}
\end{theorem}
\begin{proof}
  If $i\in T'$ then take $\alpha_i=0$ and $\beta_i=1$.
  Then apply Theorem~\ref{Pf} with $A_{i,j} = \alpha_i \beta_j \L_{i,j}$,
  and use \eqref{directed-grove} to interpret the determinants.
  The factor of $(-1)^m$ in \eqref{directed-grove} is absorbed into the Pfaffian
  because we listed each $t'_i$ before $t_i$.
\end{proof}

Any of \eqref{CIM} or \eqref{L[RGB]} or \eqref{directed-grove} can be recovered from \eqref{Pf(LL)} by choosing the $\alpha$'s and $\beta$'s suitably and/or setting $z=1$.

\subsection{Proof of main theorem} \label{sec:main}

Our approach to proving the Pfaffian formulas in Theorem~\ref{thm:Pf-udfe} is to prove that the right hand sides of \eqref{dot-sum-detL} and \eqref{bar-sum-detG} are equal as polynomials in formal variables to the Pfaffian expressions.  We will not, for example, use the fact that $\sum_j L_{i,j}=0$ or other relations that the electrical network quantities might satisfy, since the $L_{i,j}$'s and the $G_{i,j}$'s satisfy different relations.  By working with formal variables that do not satisfy these extra relations, the same proof works for both the $L$-$L'$ polynomials and the $G$-$G'$ polynomials.

The roles of the $\s$ and $\I{}$ symbols are reversed between the $L$-$L'$ polynomials for $\dddot Z[\tau]$ and the $G$-$G'$ polynomials for $\overline Z[\tau]$.  As a matter of convenience, we will give these symbols the roles they have for the $G$-$G'$ polynomials.  To obtain the $L$-$L'$ polynomials, we will at some point substitute $\I{}$ for $\s$ and $\s$ for $\I{}$.

Let $\lambda$ be a labeled augmented cyclic Dyck path with $n$ symbols.
Let $\lambda^*$ be the substring obtained from $\lambda$ by excising all $\s$ and $\I{}$ symbols that it contains, let $n^*$ be the length of $\lambda^*$, and let $E_\lambda$ be the set of labels of $\s$ symbols.
For our running example
\[
\lambda = \U{1}\S{2}\U{3}\D{4}\D{5}\I{6}\D{7}\F{8}\I{9}\U{10}\S{1\!1}\U{12}\D{13}\OD{14} \]
we have
\[
\lambda^* =
\U{1}\U{3}\D{4}\D{5}\D{7}\F{8}\U{10}\U{12}\D{13}\OD{14}\,,
\quad\quad\quad n^*=10\,,\quad\quad\quad E_\lambda=\{2,11\}\,.
\]
For an arbitrary $n\times n$ matrix $\A$ we define
\begin{equation} \label{Z_lambda}
\Zs_\lambda(\A)\colonequals
 \sum_{\substack{R^*\subset\{1,\dots,n^*\}\\|R^*|=n^*/2\\\{n^*\}\cap R=\varnothing\\S^*=\{1,\dots,n^*\}\setminus R^*}}
 \frac{\B_{\lambda^*,S^*}(e^{2t})}{(1-e^{2t})^k} \det \A_{\lambda^*(R^*),E_\lambda}^{\lambda^*(S^*),E_\lambda}
\end{equation}
Recall from \eqref{eq:A-inverse} that $\B_{\lambda,S}(\zeta)$ is a sum over $\mu\succeq\lambda$ of $(\text{\# of c.i.\ Dyck tilings of $\lambda^\circ/\mu^\circ$})$ times
$\zeta$ to the power
\begin{multline*}
\text{(\# up steps of $\mu$ in $S$ before flat step)}\\
-\text{(\# down steps of $\mu$ in $S$ after flat step) }\\
+\text{(\# down steps of $\lambda$ after flat step)}\,.
\end{multline*}
We define
\begin{multline} \label{Z^mu}
\Zs^\mu(\A)\colonequals \frac{1}{(1-e^{2t})^k}
 \sum_{\substack{R^*\subset\{1,\dots,n^*\}\\|R^*|=n^*/2\\\{n^*\}\cap R^*=\varnothing\\S^*=\{1,\dots,n^*\}\setminus R^*}}
\raisebox{-18pt}{$\displaystyle\begin{aligned}
&\det \A_{\mu^*(R^*),E_\mu}^{\mu^*(S^*),E_\mu}\times\\
&\exp\big[2t\,\big|\text{(up steps of $\mu^*$ before flat step)} \cap S^*\big|\big] \div\\
&\exp\big[2t\,\big|\text{(down steps of $\mu^*$ after flat step)} \cap S^*\big|\big]\,.
\end{aligned}$}
\end{multline}
Observe that if $\mu\succeq\lambda$ then $\mu^*(\cdot)=\lambda^*(\cdot)$ and $E_\mu=E_\lambda$, so
\begin{multline} \label{sum-Z^mu}
\Zs_\lambda(\A)=
\exp\big[2t\,(\text{\# down steps of $\lambda$ after flat step})\big]\times\\
\sum_{\mu\succeq\lambda} [\text{\# of c.i.\ Dyck tilings of $\lambda^\circ/\mu^\circ$}] \Zs^\mu(\A)\,.
\end{multline}

The following lemma will help us evaluate $\Zs^\mu(\A)$:

\newcommand{\Ac}{\widetilde{A}}
\begin{lemma} \label{lem:Pf(A[B,C,U,V])}
Suppose $n\geq0$ is even, $B,C,U,V\subset\{1,\dots,n\}$, $B\cap C=\varnothing$, and $U\cap V=\varnothing$.
Let $A$ be an arbitrary $n\times n$ matrix.
Then
\begin{multline} \label{Pf(A[B,C,U,V])}
 \sum_{\substack{R\subset\{1,\dots,n\}\\|R|=n/2\\B\subset R\\C\cap R=\varnothing\\S=\{1,\ldots,n\}\setminus R}}
  \exp\Big[2t\big(|S\cap U| -|S\cap V|\big)\Big]
  d_R(A)
=\\= \exp\big[t(|U|-|V|)\big]\times
\Pf\bordermatrix[{[]}]{
&\scriptstyle j\in B& \scriptstyle j\notin B\cup C&\scriptstyle j\in C\cr
\scriptstyle i\in B& 0&\Ac_{i,j}&\Ac_{i,j}\cr
\scriptstyle i\notin B\cup C   &-\Ac_{j,i}& \Ac_{i,j}-\Ac_{j,i} &\Ac_{i,j}\cr
\scriptstyle i\in C&-\Ac_{j,i}&-\Ac_{j,i}&0}_{\ijnn\,,}
\end{multline}
where
\[
\Ac_{i,j}=A_{i,j} \exp\big[t(1_{j\in U}-1_{i\in U}-1_{j\in V}+1_{i\in V})\big]\,.
\]
\end{lemma}

\begin{proof}
  Observe that $2|S\cap U|=|S\cap U|-|R\cap U|+|U|$, and similarly for $2|S\cap V|$. Since $\Ac$ is obtained from $A$ by multiplying the $i$th row by $\exp[t(1_{i\in V}-1_{i \in U})]$ and $j$th column by $\exp[t(1_{j\in U}-1_{j\in V})]$, the determinants $d_R(A)$ and $d_R(\Ac)$ differ by a factor depending on $R$ and $S$:
\[
\exp\Big[2t\big(|S\cap U| -|S\cap V|\big)\Big]\times
  d_R(A) = \exp\big[t(|U|-|V|)\big]\times d_R(\Ac)\,.
\]
We can set $A_{i,j}=0$ whenever $j\in B$ or $i\in C$,
since these variables do not occur in equation~\eqref{Pf(A[B,C,U,V])}.
We then remove the restrictions $B\subset R$ and $C\cap R=\varnothing$ in the
summation on the left-hand side of equation~\eqref{Pf(A[B,C,U,V])},
since with the above variables zeroed out, $d_R(A)=0$ whenever $B \not \subset R$ or $C\cap R\neq \varnothing$.  Without these restrictions on the sum, we can apply Theorem~\ref{Pf} to sum up the $d_R(\Ac)$'s to obtain~\eqref{Pf(A[B,C,U,V])}.
\end{proof}

\newcommand{\AAc}{\widetilde{\A}}
\begin{lemma} \label{Pf(M_mu(A))}
Let $\mu$ be a labeled augmented cyclic Dyck path with $n$ symbols,
and suppose $\overline\mu$ has length $m$.
Let $U$ denote the set of labels in $\overline\mu$ above $\Op$ symbols,
and let $V$ denote the set of labels in $\overline\mu$ above $\Om$ symbols.
Let $\A$ be an arbitrary $n\times n$ matrix, and let
\begin{equation}\label{A-tilde}
\AAc_{i,j}=\A_{i,j} \exp\big[t(1_{j\in U}-1_{i\in U}-1_{j\in V}+1_{i\in V})\big]\,.
\end{equation}
Then
\begin{equation} \label{Zsmu(A)}
\Zs^\mu(\A)=
\Pf\bordermatrix[{[]}]{
&  \scriptstyle\overline\mu_j\neq\Od & \scriptstyle\overline\mu_j=\Od \cr
\scriptstyle\overline\mu_i\neq\Od &
\frac{\displaystyle\AAc_{\overline\mu(i),\overline\mu(j)}-\AAc_{\overline\mu(j),\overline\mu(i)}}{\textstyle 1-e^{2t}}&\AAc_{\overline\mu(i),\overline\mu(j)}\cr\cr
\scriptstyle\overline\mu_i=\Od&
-\AAc_{\overline\mu(j),\overline\mu(i)}&0}_{\ijNN{m}\,.}
\end{equation}
\end{lemma}
\begin{proof}
Recall that $\mu^\s$ is the string obtained from $\mu$ by replacing each $\S{i}$ symbol with $\SS{i}$ and omitting each $\I{}$ symbol.
Let $B_\mu$ denote the positions of these new $\Oc$'s (replacing an~$\s$) in $\mu^\s$, let $C_\mu$ denote the positions of these new $\Od$'s in $\mu^\s$.  The strings $\mu^\s$ and $\overline\mu$ have the same length, which we are calling $m$.  If $\mu$ is our earlier example
\[\mu=\U{1}\S{2}\U{3}\D{4}\D{5}\I{6}\D{7}\F{8}\I{9}\U{10}\S{1\!1}\U{12}\D{13}\OD{14}\,,\]
then
\[
\mu^\s = \U{1}\SS{2}\U{3}\D{4}\D{5}\D{7}\F{8}\U{10}\SS{1\!1}\U{12}\D{13}\OD{14}
\]
and
\[
B_\mu=\text{positions of }\{\OC{2},\OC{11}\}=\{2,10\}
\quad\quad\text{and}\quad\quad
C_\mu=\text{positions of }\{\OD{2},\OD{11}\}=\{3,11\}\,.
\]

Let $U_\mu$ denote the set of positions at which $\mu$ has an $\U{}$
before its $\F{}$, and let $V_\mu$ denote positions at which $\mu$ has
a $\D{}$ after its $\F{}$.

For a given $\mu$, the subsets $R^*$ of $\{1,\ldots,n^*\}$ for which $|R^*|=n^*/2$ are in
straightforward bijective correspondence with those subsets $R$ of
$\{1,\ldots,m\}$ for which $|R|=m/2$, $B_\mu\subset R$ and $C_\mu\cap R=\varnothing$,
i.e., $R=\overline{\mu}^{-1}(\mu^*(R^*))\cup B_\mu$.  Consider the pairing between $R^*$ and $S^*=\{1,\ldots,n^*\}\setminus R^*$ given by the cycle lemma bijection.  This pairing naturally extends to a pairing between $R$ and $S=\{1,\ldots,m\}\setminus R$, where a pair $(r^*,s^*)$ gets mapped to the pair $(\overline{\mu}^{-1}(\mu^*(r^*)),\overline{\mu}^{-1}(\mu^*(s^*)))$, with the pairing between $R$ and $S$ also containing the pairs $(b,b+1)$ for each $b\in B_\mu$.  Provided $n^*\notin R^*$, this extended pairing is precisely the pairing between $R$ and $S$ given by the cycle lemma bijection.
Thus
\begin{equation} \label{mu*-mu+}
 \sum_{\substack{R^*\subset\{1,\dots,n^*\}\\|R^*|=n^*/2\\\{n^*\}\cap R^*=\varnothing\\S^*=\{1,\dots,n^*\}\setminus R^*}}
\raisebox{-18pt}{$\displaystyle\left(\begin{aligned}
&\det \A_{\mu^*(R^*),E_\mu}^{\mu^*(S^*),E_\mu}\times\\
&\exp\big[2t\,\big| U_{\mu^*} \cap S^*\big|\big] \div\\
&\exp\big[2t\,\big| V_{\mu^*} \cap S^*\big|\big]\
\end{aligned}\right)$}
= \sum_{\substack{R\subset\{1,\dots,m\}\\|R|=m/2\\\{m\}\cap R=\varnothing\\B_\mu\subset R\\C_\mu\cap R=\varnothing\\S=\{1,\dots,m\}\setminus R}}
\raisebox{-18pt}{$\displaystyle\left(\begin{aligned}
&\det \A_{\mu^\s(R)}^{\mu^\s(S)}\times\\
&\exp\big[2t\,\big|U_{\mu^\s}\cap S\big|\big] \div\\
&\exp\big[2t\,\big|V_{\mu^\s}\cap S\big|\big]
\end{aligned}\right)_{\,.}$}
\end{equation}
We could apply Lemma~\ref{lem:Pf(A[B,C,U,V])} with $B=B_\mu$ and $C=C_\mu\cup\{m\}$ to evaluate the right-hand side of \eqref{mu*-mu+},
but it turns out to work better with $B=\varnothing$, $C=C_\mu\cup\{m\}$.
So long as $(C_\mu\cup\{m\})\cap R=\varnothing$, if $B_\mu\not\subseteq R$, then the determinant $\det \A_{\mu^\s(R)}^{\mu^\s(S)}$ has at least one repeated column and therefore does not contribute to the sum.
Applying Lemma~\ref{lem:Pf(A[B,C,U,V])} with $B=\varnothing$, $C=C_\mu\cup\{m\}$, $U=U_{\mu^\s}$, $V=V_{\mu^\s}$, and $n=m$, and then using the fact that $\mu^\s(\cdot)=\overline\mu(\cdot)$, we see that the right-hand side of \eqref{mu*-mu+} equals
\[
\Pf\bordermatrix[{[]}]{
&  j\notin C& j\in C\cr
 i\notin C &
\AAc_{\overline\mu(i),\overline\mu(j)}-\AAc_{\overline\mu(j),\overline\mu(i)}&\AAc_{\overline\mu(i),\overline\mu(j)}\cr
 i\in C&
-\AAc_{\overline\mu(j),\overline\mu(i)}&0}_{\ijNN{m}\,,}
\]
with $\AAc$ defined as in \eqref{A-tilde}.  Observe that $C_\mu\cup\{m\}$, $U_{\mu^\s}$, and $V_{\mu^\s}$ are the locations of $\Od$, $\Op$, and $\Om$ symbols in $\overline\mu$ respectively (which is of course the reason we defined $\overline\mu$ the way we did).

The definition of $\Zs^\mu(\A)$ also contains a factor of $1/(1-e^{2t})^{n^*/2-1}$. If for some $x$ we scale the rows and columns not in $C$ by a factor of $x^{1/2}$, and scale the rows and columns in $C$ by a factor of $x^{-1/2}$, the Pfaffian is scaled by a factor of $x^{[(m-|C|)-|C|]/2}$.
Now $m=n^*+2|E|$ and $|C|=|E|+1$, so $[(m-|C|)-|C|]/2=n^*/2-1$.
Upon taking $x=1/(1-e^{2t})$, we obtain \eqref{Zsmu(A)}.
\end{proof}

So far all these calculations are exact.  Next we take the limit $t\to 0$:
\begin{lemma} \label{Pf(M_mu(AA))}
Let $\mu$ be a labeled augmented cyclic Dyck path with $n$ symbols,
and suppose $\overline\mu$ has length $m$.
Let $\A$ be an $n\times n$ matrix of formal power series for which $\A_{i,j}(t)=\A_{j,i}(-t) = A_{i,j} + A'_{i,j} t + O(t^2)$.  Then
\[
\Zs^\mu(\A) =
\Pf\phantom{\scriptstyle \overline\mu_i\neq\Od}\underbrace{\bordermatrix[{[]}]{
& \scriptstyle \overline\mu_j\neq\Od&\scriptstyle \overline\mu_j=\Od\cr
\mathllap{\scriptstyle \overline\mu_i\neq\Od} &
\left(\!\begin{aligned}
+1_{\overline\mu_i=\Op}-1_{\overline\mu_j=\Op}\\[-3pt]
-1_{\overline\mu_i=\Om}+1_{\overline\mu_j=\Om}
\end{aligned}\right)
A_{\overline\mu(i),\overline\mu(j)}-A'_{\overline\mu(i),\overline\mu(j)}&A_{\overline\mu(i),\overline\mu(j)}\cr
\mathllap{\scriptstyle \overline\mu_i=\Od}&
-A_{\overline\mu(i),\overline\mu(j)}&0}_{\mathrlap{\ijNN{m}}}}_{M_{\overline\mu}(A,A')}\ \ \ \ \ \ \raisebox{12pt}{+O(t)}
\]
\end{lemma}
\begin{proof}
Straightforward series expansion of the expression from Lemma~\ref{Pf(M_mu(A))}.
\end{proof}

It is also straightforward to extract the coefficients of higher powers of $t$ in the series expansion $\Zs^\mu(\A)$ using Lemma~\ref{Pf(M_mu(A))}.  As discussed earlier, the constant term is relevant for computing grove probabilities.  The term linear in $t$ is relevant for computing expected winding \cite{KW4}, and also depends on just the $A_{i,j}$'s and $A'_{i,j}$'s.

\begin{proof}[Proof of Theorem~\ref{thm:Pf-udfe}]
Immediate from \eqref{dot-sum-detL}, \eqref{bar-sum-detG}, \eqref{Z_lambda}, \eqref{Z^mu}, \eqref{sum-Z^mu} and Lemma~\ref{Pf(M_mu(A))}.
For the $G$-$G'$ polynomials we substitute $G$ for $A$ and $G'$ for $A'$.
For the $L$-$L'$ polynomials we first substitute $\I{}$ for $\S{}$ and $\S{}$ for $\I{}$, and then $L$ for $A$ and $L'$ for $A'$, and we absorb the factor of $(-1)^{|T|}$ from \eqref{dot-sum-detL} into the Pfaffian by writing $\II{i}$ rather than $\SS{i}$.
\end{proof}

\section{Open problems}

The coefficients in the
Pfaffian formulas in Theorem~\ref{thm:Pf-udfe} count Dyck tilings
whose lower path is $\lambda^\circ$ and whose upper path depends on
the summand.  It is known that the sum of these coefficients is the
number of increasing labelings of the planted plane tree associated
with the Dyck path $\lambda^\circ$ \cite{KMPW}.  Is there something more to
understand here?

Is there a polynomial-time algorithm for evaluating $Z[\tau]$?  For certain $\tau$'s there will be few or even just one Pfaffian, though for general $\tau$ the number of Pfaffians is exponentially large in the number of nodes.  But these Pfaffians are all closely related to one another, which suggests the possiblity that some clever linear algebra could be used to evaluate the sum without evaluating each individual Pfaffian.

\newcommand{\MRhref}[2]{\href{http://www.ams.org/mathscinet-getitem?mr=#1}{MR#2}}
\def\@rst #1 #2other{#1}
\newcommand\MR[1]{\relax\ifhmode\unskip\spacefactor3000 \space\fi
  \MRhref{\expandafter\@rst #1 other}{#1}}

\phantomsection
\pdfbookmark[1]{References}{bib}

\bibliographystyle{hmralphaabbrv}
\bibliography{pf}

\end{document}